\documentclass[12pt]{amsproc}

\usepackage{amssymb,graphicx,xy}
\xyoption{ps}\xyoption{dvips}

\newtheorem{theorem}{Theorem}[section]
\newtheorem{proposition}[theorem]{Proposition}

\theoremstyle{definition}
\newtheorem{definition}[theorem]{Definition}
\newtheorem{example}[theorem]{Example}

\numberwithin{equation}{section}

\renewcommand{\d}{\partial}
\newcommand{\e}{\epsilon}

\begin{document}

\title{Opetopes and chain complexes}

\author{Richard Steiner}

\address{School of Mathematics and Statistics\\University of Glasgow\\
University Gardens\\Glasgow\\Scotland G12 8QW}

\email{Richard.Steiner@glasgow.ac.uk}



\begin{abstract}
We give a simple algebraic description of opetopes in terms of chain complexes, and we show how this description is related to combinatorial descriptions in terms of treelike structures. More generally, we show that the chain complexes associated to higher categories generate graphlike structures. The algebraic description gives a relationship between the opetopic approach and other approaches to higher category theory. It also gives an easy way to calculate the sources and targets of opetopes. 
\end{abstract}

\maketitle

\section{Introduction} \label{S1}

There are many different approaches to the theory of higher categories; see \cite{B3},~10 for an overview. Some of these approaches have been put into a common framework based on chain complexes, for example the simplicial, cubical and finite disc approaches (see \cite{B4}, \cite{B5}, \cite{B6}). In this paper we will show how the opetopic approach fits into this framework.

Opetopes were originally defined by Baez and Dolan~\cite{B1}, and there is a slightly different definition due to Leinster \cite{B3},~7. In~\cite{B2} Kock, Joyal, Batanin and Mascari give a simple reformulation of Leinster's definition in terms of trees, and we will essentially use their form of the definition. The definition in terms of trees is easy to visualise, but the algebraic description given here makes some calculations easier. In particular it provides an easy way to compute source and taget opetopes.

The algebraic and combinatorial descriptions of opetopes are given in Sections \ref{S2} and~\ref{S3}. Section~\ref{S4} contains a general result on the chain complexes related to higher categories (Theorem~\ref{T4.1}), from which one sees that they generate families of graphlike structures. In Section~\ref{S5} we apply this result to the chain complexes related to opetopes. We show that they generate families of treelike structures, and are in fact equivalent to these families. In this way, we prove that the algebraic and combinatorial descriptions of opetopes are equivalent. Finally, in Section~\ref{S6}, we discuss subcomplexes of the chain complexes associated to opetopes, and we apply this discussion to the computation of sources and targets.

\section{An algebraic description of opetopes} \label{S2}

In this section we give a description of opetopes in terms of chain complexes with additional structure. The chain complexes concerned are free augmented directed complexes as defined in~\cite{B4}, and we first recall some of the definitions from that paper.

\begin{definition} \label{D2.1}
A \emph{free augmented directed complex} is an augmented chain complex of free abelian groups, concentrated in nonnegative dimensions, together with a distinguished basis for each chain group.
\end{definition}

Given a chain~$c$ in a free augmented directed complex, we will write
$$\d c=\d^+ c-\d^- c$$
such that $\d^+ c$ and $\d^- c$ are sums of basis elements without common terms; thus $\d^+ c$ and $\d^- c$ are the `positive and negative parts' of the boundary of~$c$.

The free augmented directed complexes of greatest general interest are those which are unital and loop-free in the following senses.

\begin{definition} \label{D2.2}
A free augmented directed complex is \emph{unital} if
$$\e(\d^-)^q a=\e(\d^+)^q a=1$$
for each basis element~$a$, where $q$~is the dimension of~$a$. 
\end{definition}

\begin{definition} \label{D2.3}
A free augmented directed complex is \emph{loop-free} if its basis elements can be partially ordered such that, for any basis element~$a$ and for any integer $r>0$, the basis elements which are terms in $(\d^-)^r a$ precede the basis elements which are terms in $(\d^+)^r a$.
\end{definition}

The complexes which represent opetopes will be loop-free and unital. They will also be \emph{atomic}; that is to say, they will be generated by single basis elements in the following sense.

\begin{definition} \label{D2.4}
A free augmented directed complex is \emph{atomic} of dimension~$n$ if it has no basis elements of dimension greater than~$n$, if it has exactly one basis element of dimension~$n$, and if every basis element of dimension less than~$n$ is a term in $\d^- a$ or $\d^+ a$ for some higher-dimensional basis element~$a$.
\end{definition}

In the opetopic approach to higher category theory, operations have single outputs and any number of inputs; in particular there are nullary operations with zero inputs. Correspondingly, for a positive-dimensional basis element~$a$, the chain $\d^+ a$ should be a single basis element, while the chain $\d^- a$ can be more or less arbitrary. But we do not want to have $\d^- a=0$, because this would be incompatible with unitality, so we have a  problem with nullary operations. To cope with this, we introduce a distinguished class of \emph{thin basis elements} of positive dimensions, which are in some sense to be regarded as negligible; nullary operations will then be represented by basis elements~$a$ such that $\d^- a$~is a thin basis element. For a positive-dimensional basis element~$a$ we then require $\d^+ a$ to be a non-thin basis element, because operations are supposed to have genuine single outputs. For a thin basis element~$a$ we will also require $\d^- a$ to be a non-thin basis element; this is because $\d a$ should be negligible, so that $\d^- a$ should be the same kind of chain as $\d^+ a$. In order to avoid redundancy, we will also require every thin basis element to be of the form $\d^- a$ for some other basis element~$a$; that is to say, we will have the minimum number of thin basis elements required for the nullary operations. It is convenient to make the definition in two stages, as follows.

\begin{definition} \label{D2.5}
An \emph{opetopic directed complex} is an atomic loop-free unital free augmented directed complex, together with a distinguished class of positive-dimensional basis elements called \emph{thin basis elements}, subject to the following conditions: 

\textup{(1)} If $a$~is a positive-dimensional basis element then $\d^+ a$ is a non-thin basis element.

\textup{(2)} If $a$~is a thin basis element then $\d^- a$ is a non-thin basis element.
\end{definition}

\begin{definition} \label{D2.6}
An opetopic directed complex is \emph{reduced} if every thin basis element is of the form $\d^- a$ for some other basis element~$a$.
\end{definition}

The main result of this paper is now as follows.

\begin{theorem} \label{T2.7}
Opetopes are equivalent to reduced opetopic directed complexes.
\end{theorem}

The conditions for unitality and loop-freeness are appropriate for free augmented directed complexes in general (see~\cite{B4}). For opetopic directed complexes they can be expressed more simply, because of the following result.

\begin{proposition} \label{P2.8}
Let $K$ be a free augmented directed complex such that $\d^+ a$ is a basis element whenever $a$~is a positive-dimensional basis element.

\textup{(1)} The complex~$K$ is unital if and only if $\e a=1$ for every zero-dimensional basis element~$a$.

\textup{(2)} The complex~$K$ is loop-free if and only if its basis elements can be partially ordered such that, for any basis element~$a$, the basis elements which are terms in $\d^- a$ precede the basis elements which are terms in $\d^+ a$.
\end{proposition}

\begin{proof}
(1) If $K$~is unital then certainly $\e a=1$ for every zero-dimensional basis element~$a$.

Conversely, suppose that $\e a=1$ for every zero-dimensional basis element~$a$. We must show that $\e(\d^-)^q b=\e(\d^+)^q b=1$ for a basis element~$b$ of arbitrary dimension~$q$. Now, it follows from the hypothesis that $(\d^+)^q b$ is a basis element; therefore $\e(\d^+)^q b=1$. From the equality $\d\d=0$ we get $\d\d^-=\d\d^+$, hence $\d^-\d^-=\d^-\d^+$, and from the equality $\e\d=0$ we similarly get $\e\d^-=\e\d^+$; therefore $\e(\d^-)^q b=\e(\d^+)^q b$, so that $\e(\d^-)^q b=1$ as well.

(2) It suffices to prove the following: if $a$~is a $q$-dimensional basis element and if $0<r\leq q$, then there is a basis element~$a'$ such that $(\d^-)^r a=\d^- a'$ and $(\d^+)^r a=\d^+ a'$. But we can take $a'=(\d^+)^{r-1}a$, which is a basis element by hypothesis; we certainly have $(\d^+)^r a=\d^+ a'$, and we also have $(\d^-)^r a=\d^- a'$ because $\d^-\d^-=\d^-\d^+$.
\end{proof}

\begin{table}

$$\begin{array}{||l|c|l|l||l|c|l|l||}
\hline\hline
a&\dim a&\d^- a&\d^+ a&b&\dim b&\d^- b&\d^+ b\\
\hline\hline
a_{17}&5&a_{16}+a_{15}&b_{17}&
b_{17}&4&a_{12}+a_{11}&b_{13}\\
&&\quad {}+a_{14}+a_{13}&&&&\quad {}+a_{10}+a_{9.5}&\\
&&&&&&\quad {}+a_9+a_8&\\ 
\hline
a_{16}&4&a_{9.5}&b_{16}&b_{16}&3&b_{12}&b_{9.5}\\
a_{15}&&a_{11}+a_9&b_{15}&b_{15}&&a_5+b_{9.5}&b_9\\
a_{14}&&a_{10}+a_8&b_{14}&b_{14}&&a_{5.5}+b_9&b_8\\
a_{13}&&a_{12}+b_{16}&b_{13}&b_{13}&&a_7+a_6&b_8\\
&&\quad {}+b_{15}+b_{14}&&&&\quad {}+a_{5.5}+a_5&\\ 
\hline
a_{12}&3&a_7+a_6&b_{12}&b_{12}&2&a_3+a_2&b_6\\
a_{11}&&a_5&b_{11}&b_{11}&&b_6+b_{5.5}&b_5\\
a_{10}&&a_{5.5}&b_{10}&b_{10}&&a_4&b_{5.5}\\
a_{9.5}&&b_{12}&b_{9.5}&b_{9.5}&&a_3+a_2&b_6\\
a_9&&b_{11}+b_{9.5}&b_9&b_9&&a_3+a_2+b_{5.5}&b_5\\
a_8&&b_{10}+b_9&b_8&b_8&&a_4+a_3+a_2&b_5\\
\hline
a_7&2&a_3&b_7&b_7&1&b_4&b_3\\
a_6&&a_2+b_7&b_6&b_6&&b_4&b_2\\
a_{5.5}&&a_4&b_{5.5}&b_{5.5}&&a_1&b_4\\
a_5&&b_6+b_{5.5}&b_5&b_5&&a_1&b_2\\
\hline
a_4&1&a_1&b_4&b_4&0&0&0\\
a_3&&b_4&b_3&b_3&&0&0\\
a_2&&b_3&b_2&b_2&&0&0\\
\hline
a_1&0&0&0&&&&\\
\hline\hline
\end{array}$$ 
\caption{A reduced opetopic directed complex}

\end{table}

\begin{example} \label{E2.9}
We will now describe the reduced opetopic directed complex~$K$ corresponding to the $5$-dimensional opetope of \cite{B2},~5.9, which is shown in Figure~2 (this example is discussed further at the end of section~\ref{S3}). In this opetope there are dots and spheres numbered $1$~to~$16$ (the spheres being drawn as circles), and in~$K$ there are corresponding basis elements~$a_i$ for $1\leq i\leq 16$. These are the basis elements of dimension less than~$5$ which are not thin and are not of the form $\d^+ a$ for any basis element~$a$. The remaining basis elements are as follows: a $5$-dimensional basis element denoted~$a_{17}$; two thin basis elements denoted $a_{5.5}$~and~$a_{9.5}$; the basis elements $\d^+ a_i$ for $\dim a_i>0$, which are denoted~$b_i$. The dimensions and boundaries of the basis elements are shown in Table~1. The augmentation is given by
$$\e a_1=\e b_4=\e b_3=\e b_2=1.$$
The only thin basis elements are $a_{9.5}$~and~$a_{5.5}$.

Note that the terms of $\d^- a_i$ are of the form~$a_j$ or of the form~$b_k$ with $k>i$, and note also that $\d b_i=\d\d^- a_i$ for $\dim a_i>0$. From these observations, $\d\d a_i=0$ for all~$i$, and it then follows by downward induction on~$i$ that $\d\d b_i=0$ for $\dim a_i>0$. It is now easy to see that $K$~is a free augmented directed complex, and it follows from Proposition~\ref{P2.8}(1) that $K$~is unital. For any basis element~$c$, the terms of $\d^- c$ come before those of $\d^+ c$ in the list
$$a_{17},\ldots,a_1,b_{17},\ldots,b_2;$$
hence, by Proposition~\ref{P2.8}(2), $K$~is also loop-free. It is is now easy to see that $K$~really is a reduced opetopic directed complex.
\end{example}

\section{A combinatorial description of opetopes} \label{S3}

We will now recall the combinatorial definition of opetopes given in \cite{B2}. This definition involves trees whose edges and vertices approximately correspond to the basis elements in the associated reduced opetopic directed complex. We will reformulate the definition so that the correspondence becomes exact; this makes the theory easier to generalise and also seems to make it somewhat simpler. 

The basic concept in~\cite{B2} is a sequence of \emph{subdivided trees} and \emph{constellations} (see \cite{B2},~1.19). A subdivided tree here means a connected and simply connected finite graph together with certain distinguished sets of vertices as follows: there is a vertex of valency~$1$ called the \emph{output vertex}; every other vertex of valency~$1$ is either an \emph{input vertex} or a \emph{null-dot}; there is a distinguished class of vertices of valency~$2$ called \emph{white dots}. We will regard these subdivided trees as directed graphs such that there are directed paths to the output vertex from every other vertex (there is clearly a unique directed graph structure with this property); in this way, for every edge~$e$, one of the end-points of~$e$ becomes its \emph{source} and the other end-point becomes its \emph{target}. We will then remove the output vertex and the input vertices; the result is like a directed graph except that there may be edges with only one end-point, and there may even be edges with no end-points at all; a structure of this kind will be called a \emph{network}. To each null-dot~$v$ we then add an edge with target~$v$ and with no source. The result of all this is a network with a distinguished class of \emph{thin edges} (the edges added to the null-dots) and a distinguished class of \emph{thin vertices} (the white-dots). In these terms, the definition of opetopes comes out as follows.

\begin{definition} \label{D3.1}
A \emph{network} is a structure consisting of two finite sets $E$~and~$V$, together with two subsets $I$~and~$O$ of~$E$ and two functions
$$s\colon E\setminus I\to V,\quad t\colon E\setminus O\to V.$$
The members of~$E$ are called \emph{edges}, the members of~$V$ are called \emph{vertices}, the members of~$I$ are called \emph{input edges}, and the members of~$O$ are called \emph{output edges}. If $e$~is an edge which is not an input edge, then $s(e)$ is called the \emph{source} of~$e$; if $e$~is an edge which is not an output edge, then $t(e)$ is called the \emph{target} of~$e$.
\end{definition}

\begin{definition} \label{D3.2}
A \emph{path} in a network from an edge~$e'$ to an edge~$e$ is a sequence
$$e_0,v_1,e_1,\dots,v_k,e_k,$$
consisting of edges~$e_r$ and vertices~$v_r$, such that $v_r$ is the target of~$e_{r-1}$ and the source of~$e_r$ and such that $e_0=e'$, $e_k=e$.

A network is \emph{acyclic} if contains no path from any edge~$e$ to itself other than the single-term path~$e$.

A network is \emph{linear} if it is an acyclic network consisting of a single path.
\end{definition}

\begin{definition} \label{D3.3}
An \emph{opetopic network} is an acyclic network with a unique output edge, together with distinguished sets of edges and vertices called \emph{thin edges} and \emph{thin vertices}, satisfying the following conditions.

\textup{(1)} Every thin edge is an input edge.

\textup{(2)} Every thin vertex is the target for a unique edge, and this edge is not thin.
\end{definition}

\begin{definition} \label{D3.4}
An opetopic network is \emph{reduced} if every thin edge has a target vertex, and this vertex is not the target for any other edge.
\end{definition}

\begin{definition} \label{D3.5}
Let $N$~and~$P$ be opetopic networks. A \emph{constellation} from~$N$ to~$P$ is a bijection~$c$ from the set of vertices in~$N$ to the set of input edges in~$P$ such that the thin vertices of~$N$ correspond to the thin edges of~$P$ and such that the following condition holds: if $e$~is an edge in~$P$ and if $I_e$ is the set of input edges~$e'$ in~$P$ such that there is a path in~$P$ from~$e'$ to~$e$, then there is exactly one edge in~$N$ with a source but not a target in $c^{-1}(I_e)$.
\end{definition}

\begin{definition} \label{D3.6}
An $n$-dimensional \emph{opetopic sequence} is a sequence of opetopic networks and constellations
$$N_0\to N_1\to\ldots\to N_n$$
such that $N_0$~is linear with no thin edges and such that $N_n$~consists of a single edge.
\end{definition}

\begin{definition} \label{D3.7}
An $n$-dimensional \emph{opetope} is a sequence of reduced opetopic networks and constellations
$$N_0\to N_1\to\ldots\to N_n$$
such that $N_0$~is linear with no thin edges and such that $N_n$~consists of a single edge.
\end{definition}

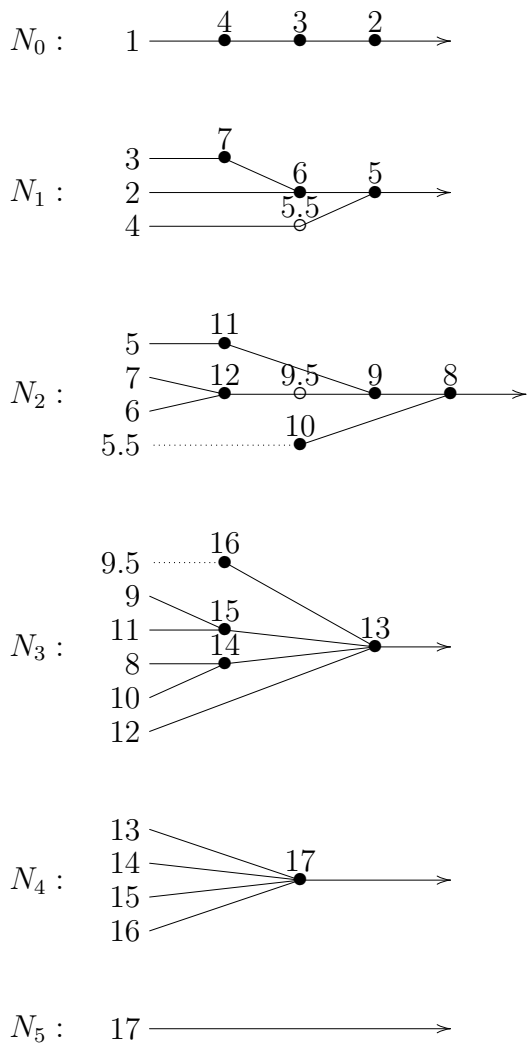
\begin{figure} 
\begin{align*}
&\quad\begin{xy}
0;<0.5cm,0cm>:
(-3,0) *{N_0:},
(0,0) *+!R{1},
(2,0) *{\bullet} *+!D{4},
(4,0) *{\bullet} *+!D{3},
(6,0) *{\bullet} *+!D{2},
(0,0); (2,0) **@{-}, 
(2,0); (4,0) **@{-}, 
(4,0); (6,0) **@{-}, 
(6,0); (8,0) **@{-} *\dir{>}
\end{xy}\\\\
&\quad\begin{xy}
0;<0.5cm,0cm>:
(-3,0.9) *{N_1:},
(0,1.8) *+!R{3},
(0,0.9) *+!R{2},
(0,0) *+!R{4},
(2,1.8) *{\bullet} *+!D{7},
(4,0.9) *{\bullet} *+!D{6},
(4,0) *{\circ} *+!D{5.5},
(6,0.9) *{\bullet} *+!D{5},
(0,1.8); (2,1.8) **@{-}, 
(0,0.9); (4,0.9) **@{-}, 
(0,0); (4,0) **@{-}, 
(2,1.8); (4,0.9) **@{-}, 
(4,0.9); (6,0.9) **@{-}, 
(4,0); (6,0.9) **@{-}, 
(6,0.9); (8,0.9) **@{-} *\dir{>}
\end{xy}\\\\
&\quad\begin{xy}
0;<0.5cm,0cm>:
(-3,1.35) *{N_2:},
(0,2.7) *+!R{5},
(0,1.8) *+!R{7},
(0,0.9) *+!R{6},
(0,0) *+!R{5.5},
(2,2.7) *{\bullet} *+!D{11},
(2,1.35) *{\bullet} *+!D{12},
(4,1.35) *{\circ} *+!D{9.5},
(4,0) *{\bullet} *+!D{10},
(6,1.35) *{\bullet} *+!D{9},
(8,1.35) *{\bullet} *+!D{8},
(0,2.7); (2,2.7) **@{-},
(0,1.8); (2,1.35) **@{-},
(0,0.9); (2,1.35) **@{-},
(0,0); (4,0) **@{.},
(2,2.7); (6,1.35) **@{-},
(2,1.35); (4,1.35) **@{-},
(4,1.35); (6,1.35) **@{-},
(4,0); (8,1.35) **@{-},
(6,1.35); (8,1.35) **@{-},
(8,1.35); (10,1.35) **@{-} *\dir{>}
\end{xy}\\\\
&\quad\begin{xy}
0;<0.5cm,0cm>:
(-3,2.25) *{N_3:},
(0,4.5) *+!R{9.5},
(0,3.6) *+!R{9},
(0,2.7) *+!R{11},
(0,1.8) *+!R{8},
(0,0.9) *+!R{10},
(0,0) *+!R{12},
(2,4.5) *{\bullet} *+!D{16},
(2,2.7) *{\bullet} *+!D{15},
(2,1.8) *{\bullet} *+!D{14},
(6,2.25) *{\bullet} *+!D{13},
(0,4.5); (2,4.5) **@{.},  
(0,3.6); (2,2.7) **@{-}, 
(0,2.7); (2,2.7) **@{-},  
(0,1.8); (2,1.8) **@{-},  
(0,0.9); (2,1.8) **@{-},  
(0,0); (6,2.25) **@{-},  
(2,4.5); (6,2.25) **@{-},  
(2,2.7); (6,2.25) **@{-},  
(2,1.8); (6,2.25) **@{-},  
(6,2.25); (8,2.25) **@{-} *\dir{>}  
\end{xy}\\\\
&\quad\begin{xy}
0;<0.5cm,0cm>:
(-3,1.35) *{N_4:},
(0,2.7) *+!R{13},
(0,1.8) *+!R{14},
(0,0.9) *+!R{15},
(0,0) *+!R{16},
(4,1.35) *{\bullet} *+!D{17},
(0,2.7); (4,1.35) **@{-},  
(0,1.8); (4,1.35) **@{-},  
(0,0.9); (4,1.35) **@{-},  
(0,0); (4,1.35) **@{-},  
(4,1.35); (8,1.35) **@{-} *\dir{>}  
\end{xy}\\\\
&\quad\begin{xy}
0;<0.5cm,0cm>:
(-3,0) *{N_5:},
(0,0) *+!R{17},
(0,0); (8,0) **@{-} *\dir{>}
\end{xy}
\end{align*}
\caption{An opetope represented as a sequence of networks}
\end{figure}

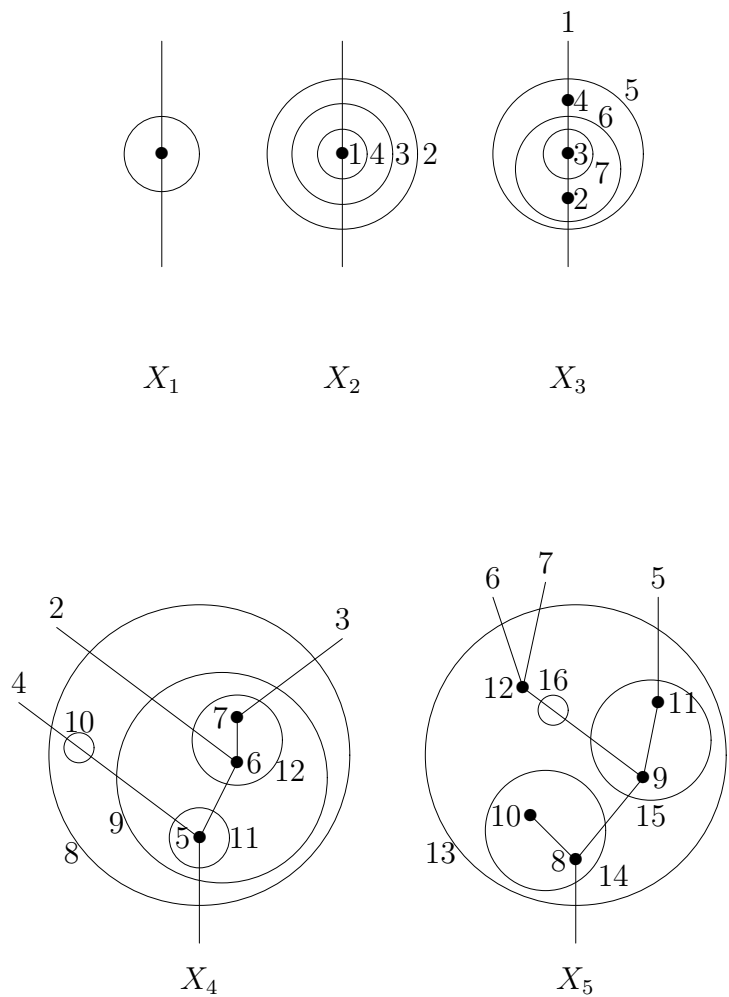
\begin{figure} 
$$\begin{xy}
0;<1cm,0cm>:
(7.5,11) *\cir<0.5cm>{},
(7.5,11) *{\bullet},
(7.5,12.5); (7.5,9.5) **@{-},
(9.9,11) *\cir<1cm>{} *\cir<0.67cm>{} *\cir<0.33cm>{},
(9.9,11) *{\bullet}, (10.07,11) *{1},
(9.9,12.5); (9.9,9.5) **@{-},
(11.07,11) *{2},
(10.7,11) *{3},
(10.36,11) *{4},
(12.9,11) *\cir<1cm>{} *\cir<0.33cm>{},
(12.9,10.8) *\cir<0.7cm>{},
(12.9,11.7) *{\bullet}, (13.07,11.7) *{4},
(12.9,11) *{\bullet}, (13.07,11) *{3},
(12.9,10.4) *{\bullet}, (13.07,10.4) *{2},
(12.9,12.5); (12.9,9.5) **@{-},
(12.9,12.5) *+!D{1},
(13.55,11.65), (13.75,11.85) *{5},
(13.2,11.3), (13.4,11.5) *{6},
(13.05,10.85), (13.35,10.75) *{7},
(8,3) *\cir<2cm>{},
(8.3,2.7) *\cir<1.4cm>{},
(6.4,3.1) *\cir<0.2cm>{},
(8.5,3.2) *\cir<0.6cm>{},
(8,1.9) *\cir<0.4cm>{},
(8,1.9) *{\bullet} *+!R{5},
(8.5,2.9) *{\bullet} *+!L{6},
(8.5,3.5) *{\bullet} *+!R{7},
(8,1.9); (8,0.5) **@{-},
(8,1.9); (8.5,2.9) **@{-},
(8.5,2.9); (8.5,3.5) **@{-},
(8,1.9); (5.6,3.7) **@{-},
(8.5,2.9); (6.1,4.7) **@{-},
(8.5,3.5); (9.9,4.55) **@{-},
(5.6,3.7) *+!D{4},
(6.1,4.7) *+!D{2},
(9.9,4.55) *+!D{3},
(6.3,1.7) *{8},
(6.9,2.1) *{9},
(6.4,3.45) *{10},
(8.6,1.9) *{11},
(9.2,2.8) *{12},
(13,3) *\cir<2cm>{},
(12.6,2) *\cir<0.8cm>{},
(14,3.2) *\cir<0.8cm>{},
(12.7,3.6) *\cir<0.2cm>{},
(13,1.6) *{\bullet} *+!R{8},
(13.9,2.7) *{\bullet} *+!L{9},
(12.4,2.2) *{\bullet} *+!R{10},
(14.1,3.7) *{\bullet} *+!L{11},
(12.3,3.9) *{\bullet} *+!R{12},
(13,1.6); (13,0.5) **@{-},
(13,1.6); (12.4,2.2) **@{-},
(13,1.6); (13.9,2.7) **@{-},
(13.9,2.7); (12.3,3.9) **@{-},
(13.9,2.7); (14.1,3.7) **@{-},
(14.1,3.7); (14.1,5.1) **@{-},
(12.3,3.9); (11.9,5.1) **@{-},
(12.3,3.9); (12.6,5.3) **@{-},
(14.1,5.1) *+!D{5},
(11.9,5.1) *+!D{6},
(12.6,5.3) *+!D{7},
(11.2,1.7) *{13},
(13.5,1.4) *{14},
(14,2.2) *{15},
(12.7,4) *{16},
(7.5,8) *{X_1},
(9.9,8) *{X_2},
(12.9,8) *{X_3},
(8,0) *{X_4},
(13,0) *{X_5}
\end{xy}$$
\caption{An opetope represented as a sequence of trees (extension of \cite{B2}, 5.9)}
\end{figure} 

We can represent opetopes by sequences of diagrams as in Figure~1, which shows the $5$-dimensional opetope of \cite{B2},~5.9 in terms of networks; it corresponds to the reduced opetopic directed complex of Example~\ref{E2.9}. Each vertex is the target of the edges to its left and the source of the edge to its right, so that the input edges are on the left and the output edges on the right. The thin edges and vertices are shown by dotted lines and hollow dots, and the constellations are shown by the numberings of the vertices and input edges. The constellation condition works out as follows: if $e$~is an input edge in~$N_{q+1}$ labelled~$i$, then the corresponding edge in~$N_q$ is the edge with source labelled~$i$; if $e$~is a non-input edge in~$N_{q+1}$ with source labelled~$i$, then the corresponding edge in~$N_q$ is the edge with source labelled~$j$ such that $\d^+ b_i=b_j$ in Table~1.

We will now show that Definition~\ref{D3.7} agrees with the definition in~\cite{B2}. First, as an example, we compare the networks of Figure~1 with the diagrams of~\cite{B2}, 5.9. These diagrams are as shown in Figure~2, except that the rather trivial diagram~$X_1$ is not actually printed in~\cite{B2}. Each diagram~$X_q$ represents two trees $T^*_q$~and~$T_q$ as follows: the tree~$T^*_q$~is the obvious tree consisting of the dots and edges in~$X_q$; the tree~$T_q$ has input edges and vertices corresponding to the dots and circles in~$X_q$, such that there is a path from an input edge to a vertex in~$T_q$ if and only if the corresponding dot is inside the corresponding circle in~$X_q$. One finds that $T^*_q$~is isomorphic to~$T_{q-1}$, so that the five diagrams $X_1,\ldots,X_5$ represent six trees $T_0,\ldots,T_5$ in a somewhat redundant way. The trees $T_2,T_3,T_4,T_5$ evidently correspond to the networks $N_0,N_1,N_2,N_3$, with the empty circles $10$~and~$16$ corresponding to the thin vertices $5.5$~and~$9.5$. The trees $T_0$~and~$T_1$ are redundant, because they must each have a single interior vertex and a single input edge. The network~$N_5$ is similarly redundant, because it must consist of a single edge, and the network~$N_4$ is then determined by $N_3$~and~$N_5$. The networks in Figure~1 are therefore equivalent to the diagrams in Figure~2.

We will now describe this equivalence in general. We first observe that the networks appearing in opetopes are \emph{confluent networks} in the following sense.

\begin{definition} \label{D3.8}
A \emph{confluent} network is an acyclic network with a unique output edge such that each vertex is the source of exactly one edge and the target of at least one edge.
\end{definition}

\begin{proposition} \label{P3.9}
Every network in an opetopic sequence is confluent.
\end{proposition}

\begin{proof}
Consider an opetopic sequence $N_0\to\ldots\to N_n$. By definition, each network~$N_q$ is acyclic with a unique output edge. There are no vertices in~$N_n$. If $v$~is a vertex in~$N_0$ then $v$~is the source for exactly one edge and the target for exactly one edge, because $N_0$~is linear. If $v$~is a vertex in~$N_q$ with $0<q<n$, then $v$~corresponds to an input edge~$e$ in~$N_{q+1}$, the constellation condition applied to~$e$ shows that $v$~is the source for exactly one edge~$e'$, and the constellation condition applied to~$e'$ shows that $v$~is the target for at least one edge. In all cases, it follows that $N_q$~is confluent.
\end{proof}

It follows from Proposition~\ref{P3.9} that an opetope as defined in Definition~\ref{D3.7} is a sequence of confluent reduced opetopic networks and constellations. Sequences of confluent reduced opetopic networks and constellations correspond to sequences of subdivided trees and constellations as described in \cite{B2},~1.19. We will now make a precise comparison by considering various sequences of confluent reduced opetopic networks and constellations as follows.

Let $E$ be an opetopic network consisting of a single non-thin edge. An $n$-dimensional opetope as defined in Definition~\ref{D3.7} is a sequence $N_0\to\ldots\to N_n$ such that $N_0$~is linear with no thin edges and such that $N_n\cong E$.

Let $L$ be a linear opetopic network consisting of two non-thin edges and a single non-thin vertex. For an arbitrary confluent reduced opeto\-pic network~$N$ there can be at most one sequence of constellations 
$L\to L\to N$, 
and such a sequence exists if and only if $N$~is linear with no thin edges. An $n$-dimensional opetope as defined in Definition~\ref{D3.7} is therefore equivalent to a sequence $N_{-2}\to\ldots\to N_n$ such that $N_q\cong L$ for $q<0$ and such that $N_n\cong E$. 

Now let $L'$ be a linear opetopic network consisting of two non-thin edges and a single thin vertex. For an arbitrary confluent reduced opetopic network~$N$, there can be up to isomorphism at most one sequence $N\to N'\to E$ with $N'$ confluent and reduced: indeed, $N'$~must consist of a single non-thin vertex, a single non-thin output edge, and of input edges corresponding to the vertices of~$N$. A sequence of this kind exists if  $N\cong L'$ or if $N$~has no thin vertices and at least one non-thin vertex, but not otherwise. An $n$-dimensional opetope as defined in Definition~\ref{D3.7} is therefore equivalent to a sequence $N_{-2}\to\ldots\to N_{n-2}$ such that $N_q\cong L$ for $q<0$ and such that $N_{n-2}\cong L'$ or $N_{n-2}$ has no thin vertices and at least one non-thin vertex.

Finally, for an arbitrary confluent reduced opetopic network~$N$, there is a constellation $N\to L'$ if and only if there is a constellation $N\to E$, and there is at most one such constellation in each case. An $n$-dimensional opetope as defined in Definition~\ref{D3.7} is therefore equivalent to a sequence $N_{-2}\to\ldots\to N_{n-2}$ such that $N_q\cong L$ for $q<0$ and such that $N_{n-2}$~has no thin vertices. These sequences exactly correspond to $n$-dimensional opetopes $T_0\to\ldots\to T_n$ as defined in \cite{B2}, 1.19--1.20.

\section{Networks associated to augmented directed complexes} \label{S4}

In this section we give a general result showing that loop-free unital free augmented directed complexes generate families of networks. We use the functor~$\nu$ from free  augmented directed complexes to strict $\omega$-categories, which is fully faithful on the subcategory of loop-free unital free augmented directed complexes (see~\cite{B4}).

Let $K$ be a free augmented directed complex. Recall that the members of $\nu K$ are the double sequences
$$(\,x_0^-,x_0^+\mid x_1^-,x_1^+\mid\ldots\,)$$
such that $x_q^-$~and~$x_q^+$ are sums of $q$-dimensional basis elements, such that finitely many terms of the double sequence are non-zero, such that 
$$\e x_0^-=\e x_0^+=1,$$
and such that
$$\d x_q^-=\d x_q^+=x_{q-1}^+-x_{q-1}^-$$
for $q>0$.

Now let
$$x=(\,x_0^-,x_0^+\mid\ldots\,)$$
be a member of $\nu K$, let $q$ be a nonnegative integer, and let $\alpha$ be a sign. We write  $x_{q+1}^\alpha$ as a sum of basis elements, say
$$x_{q+1}^\alpha=a_1+\ldots+a_k,$$
and we then define a $q$-chain $g_q^\alpha(x)$ by the formulae
$$g_q^\alpha(x)
 =x_q^-+\d^+a_1+\ldots+\d^+a_k
 =\d^-a_1+\ldots+\d^-a_k+x_q^+$$
(the two expressions for $g_q^\alpha(x)$ take the same value because $\d x_{q+1}^\alpha=x_q^+-x_q^-$).

The main result is now as follows.

\begin{theorem} \label{T4.1}
Let $K$ be a loop-free unital free augmented directed complex.

\textup{(1)} If $x=(\,x_0^-,x_0^+\mid\ldots\,)$ is a member of $\nu K$, then $x_0^-$~and~$x_0^+$ are basis elements.

\textup{(2)} If $x=(\,x_0^-,x_0^+\mid\ldots\,)$ is a member of $\nu K$ and if $c$~is $x_q^-$, $x_q^+$, $g_q^-(x)$ or $g_q^+(x)$, then $c$~is a sum of distinct basis elements.
\end{theorem}

\begin{proof}
(1) For each sign~$\alpha$ we can write~$x_0^\alpha$ as a sum of basis elements, say
$$x_0^\alpha=b_1+\ldots+b_l.$$
We have $\e x_0^\alpha=1$ because $x\in\nu K$, and we have $\e b_r=1$ for $1\leq r\leq l$ because $K$~is unital. Therefore $l=1$, which means that $x_0^\alpha$~is a basis element.

(2) We will use induction on~$q$. For each value of~$q$ we first consider chains of the form~$x_q^\alpha$ and then chains of the form $g_q^\alpha(x)$.

Suppose that $c=x_q^\alpha$ with $q=0$. Then $c$~is a basis element by~(1), so that $c$~is a sum of distinct basis elements in a somewhat trivial way.

Suppose that $c=x_q^\alpha$ with $q>0$. We can write~$x_q^\alpha$ as a sum of basis elements, say
$$x_q^\alpha=b_1+\ldots+b_l.$$
By unitality the chains $\d^+ b_r$ are non-zero, so they are non-empty sums of basis elements. The sum $\d^+ b_1+\ldots+\d^+ b_l$ is part of the sum $g_{q-1}^\alpha(x)$, which is a sum of distinct basis elements by the inductive hypothesis, so the terms $\d^+b_1,\ldots,\d^+b_l$ are distinct. Therefore the terms $b_1,\ldots,b_l$ are distinct, so that $x_q^\alpha$~is a sum of distinct basis elements.

Now let $c=g_q^\alpha(x)$. Because $K$~is loop-free we can write the chain~$x_{q+1}^\alpha$ as a sum of $(q+1)$-dimensional basis elements, say 
$$x_{q+1}^\alpha=a_1+\ldots+a_k,$$
such that $i<j$ whenever the sums of $q$-dimensional basis elements $\d^+a_i$ and $\d^- a_j$ have a common term. We then have
$$g_q^\alpha(x)
 =x_q^-+\d^+ a_1+\ldots+\d^+ a_k
 =\d^- a_1+\ldots+\d^- a_k+x_q^+.$$
For $0\leq r\leq k$, let
$$y_r=x_q^-+\d a_1+\ldots+\d a_r,$$
so that in particular
$$y_0=x_q^-,\ y_k=x_q^+.$$
For a $q$-dimensional basis element~$b$, let $\lambda_r^b$ be the coefficient of~$b$ in~$y_r$. The ordering of the basis elements~$a_r$ ensures that for some~$m$ we have
$$\textup{$\lambda_0^b\leq\lambda_1^b\leq\ldots\leq\lambda_m^b$ and  $\lambda_m^b\geq\lambda_{m+1}^b\geq\ldots\geq\lambda_k^b$.}$$
Since $y_0$~and~$y_k$ are sums of basis elements, so also is~$y_r$ for $0\leq r\leq k$. We also see that the coefficient of~$b$ in $g_q^\alpha(x)$ is~$\lambda_m^b$. It therefore suffices to show that each chain~$y_r$ is a sum of distinct basis elements. 

We now observe that $\d y_r=\d x_q^-=x_{q-1}^+-x_{q-1}^-$ in the case $q>0$, and that $\e y_r=\e x_q^-=1$ in the case $q=0$. It follows that there is an element of $\nu K$ given by
$$(\,x_0^-,x_0^+\mid\ldots
 \mid x_{q-1}^-,x_{q-1}^+\mid y_r,y_r\mid 0,0\mid\ldots\,).$$
By applying what we have already proved to this member of $\nu K$ we see that $y_r$~is a sum of distinct basis elements. This completes the proof.
\end{proof}

We can express this result in terms of networks.

\begin{theorem} \label{T4.2}
Let 
$$x=(\,x_0^-,x_0^+\mid\ldots\,)$$
be a member of $\nu K$, where $K$~is a loop-free unital free augmented directed complex.

\textup{(1)} There are networks $G_q^-(x)$ and $G_q^+(x)$ for $q\geq 0$ as follows: the edges of $G_q^\alpha(x)$ correspond to the $q$-dimensional basis elements which are terms in $g_q^\alpha(x)$. The vertices in $G_q^\alpha(x)$ correspond to the $(q+1)$-dimensional elements which are terms in~$x_{q+1}^\alpha$. If a vertex~$v$ in $G_q^\alpha(x)$ corresponds to a basis element~$a$, then the edges with target~$v$ correspond to the terms of $\d^- a$ and the edges with source~$v$ correspond to the terms of $\d^+ a$.

\textup{(2)} There exists~$n$ such that $G_q^\alpha(x)$ is empty for all $q>n$.

\textup{(3)} The networks $G_q^\alpha(x)$ are acyclic for all $q\geq 0$.

\textup{(4)} The networks $G_0^\alpha(x)$ are linear.

\textup{(5)} The input edges of $G_{q+1}^\alpha$ correspond to the terms of~$x_{q+1}^-$, and the output edges of $G_{q+1}^\alpha$ correspond to the terms of~$x_{q+1}^+$. These correspondences induce bijections from the vertices of $G_q^-(x)$ to the input edges of $G_{q+1}^\alpha$ and from the vertices of $G_q^+(x)$ to the output edges of $G_{q+1}^\alpha$.
\end{theorem}

\begin{proof}
Parts (1), (2) and~(5) are obvious, and part~(3) is a consequence of the loop-freeness of~$K$. Part~(4) holds because $x_0^-$~and~$x_0^+$ are basis elements and because, by unitality, if $a$~is a $1$-dimensional basis element then $\d^- a$ and $\d^+ a$ are basis elements.
\end{proof}

\section{Opetopic directed complexes and opetopic sequences} \label{S5}

We will now apply the results of the previous section to opetopic directed complexes, and we will show that they are equivalent to opetopic sequences. We will also show that reduced opetopic directed complexes are equivalent to opetopes, thereby proving Theorem~\ref{T2.7}. In order to follow the argument, one should compare the reduced opetopic directed complex in Table~1 with the opetope in Figure~1.

Recall that an opetopic directed complex~$K$ is in particular atomic and unital. Given such a complex, we define a particular member~$\langle K\rangle$ of $\nu K$ as follows.

\begin{definition} \label{D5.1}
Let $K$ be an $n$-dimensional atomic unital free augmented directed complex with $n$-dimensional basis element~$g$. Then the \emph{canonical atom} of~$K$ is the member~$\langle K\rangle$ of $\nu K$ given by
$$\langle K\rangle=\bigl(\,(\d^-)^q g,(\d^+)^q g\mid\ldots
 \mid\d^- g,\d^+ g\mid g,g\mid 0,0\mid\ldots\,\bigr).$$
\end{definition}

(It is straightforward to check that this double sequence really is a member of $\nu K$; in particular we have $\e(\d^-)^q g=\e(\d^+)^q g=1$ because $K$~is unital.)

We will now show how opetopic sequences are associated to opetopic directed complexes.

\begin{theorem} \label{T5.2}
Let $K$ be an $n$-dimensional opetopic directed complex, and for $0\leq q\leq n$, let $N_q=G_q^-(\langle K\rangle)$. Then there is an $n$-dimensional opetopic sequence $N_0\to\ldots\to N_n$ such that the thin edges and vertices in~$N_q$ correspond to the thin basis elements in~$K$ and such that the constellations are the bijections from vertices to input edges induced by the structure of~$\langle K\rangle$.
\end{theorem}

\begin{proof}
Let $g$ be the $n$-dimensional basis element of~$K$. By Theorem~\ref{T4.2}, each network~$N_q$ is acyclic with output edges corresponding to the terms of $(\d^+)^{n-q}g$. Since $K$~is opetopic, the chains $(\d^+)^{n-q}g$ are basis elements; the networks~$N_q$ therefore have unique output edges. For the same reason, each vertex is the source for exactly one edge. Since $\d^+a$ is a non-thin basis element whenever $a$~is a basis element of positive dimension, it follows that every thin edge is an input edge. Since $\d^- a$ is a non-thin basis element whenever $a$~is a thin basis element, it follows that each thin vertex is the target for exactly one edge, and that this edge is not thin. Therefore each network~$N_q$ is an opetopic network.
By Theorem~\ref{T4.2}, $N_0$~is linear. Since thin basis elements are of positive dimension, $N_0$~has no thin edges. The network~$N_n$ consists of the single edge corresponding to~$g$. Since every thin edge is an input edge, the thin vertices of~$N_q$  correspond to the thin edges of~$N_{q+1}$ under the induced bijection. It remains to verify the constellation condition of Definition~\ref{D3.5}.

To do this, let $e$ be an edge in~$N_{q+1}$ with $q\geq 0$, let $I_e$ be the set of input edges~$e'$ in~$N_{q+1}$ such that there is a path from~$e'$ to~$e$, let $b$ be the $(q+1)$-dimensional basis element corresponding to~$e$, let $b'$ be the sum of the $(q+1)$-dimensional basis elements corresponding to the members of~$I_e$, and let $c$ be the sum of the $(q+2)$-dimensional basis elements corresponding to the vertices on the paths from members of~$I_e$ to~$e$. Since each vertex is the source for exactly one edge, every edge incident with a vertex corresponding to a term of~$c$ is in the union of these paths. From the structure of~$N_{q+1}$, it follows that $\d c=b-b'$. From this it follows that $\d b'=\d b$, hence $\d^+ b'=\d^+ b$. Since $\d^+ b$ is a basis element, so also is $\d^+ b'$. 

Now let $V_e$ be the set of vertices in~$N_q$ corresponding to the edges in~$I_e$, so that $b'$~is the sum of the basis elements corresponding to the members of~$V_e$. From the structure of~$N_q$, we see that the terms of $\d^+ b'$ correspond to the edges with a source but not a target in~$V_e$, and it follows that there is exactly one such edge.

This completes the proof.
\end{proof}

We will now prove Theorem~\ref{T2.7} by proving the following result.

\begin{theorem} \label{T5.3}
Every $n$-dimensional opetopic sequence
$$N_0\to\ldots\to N_n$$
is isomorphic to the sequence
$$G_0^-(\langle K\rangle)\to\ldots\to G_n^-(\langle K\rangle)$$
for some $n$-dimensional opetopic directed complex~$K$, and $K$~is uniquely determined up to isomorphism. The complex is reduced if and only if the sequence is an $n$-dimensional opetope.
\end{theorem}

\begin{proof} 
We first construct a suitable opetopic directed complex~$K$ as follows. 

There are $q$-dimensional basis elements only for $0\leq q\leq n$. The $q$-dimensional basis elements of~$K$ correspond to the edges of~$N_q$, and the thin basis elements correspond to the thin edges.

To define the boundary homomorphism, let $b$ be a $(q+1)$-di\-men\-sion\-al basis element with $q\geq 0$, let $e$ be the corresponding edge in~$N_{q+1}$, let $I_e$ be the set of input edges~$e'$ in~$N_{q+1}$ such that there is a path from~$e'$ to~$e$, and let $V_e$ be the corresponding set of vertices in~$N_q$. Then $\d b=\d^+ b-\d^- b$, where $\d^+ b$ is the basis element corresponding to the edge with a source but no target in~$V_e$, and where $\d^- b$ is the sum of the basis elements corresponding to the edges with a target but no source in~$V_e$.

We will now show that $K$~is a chain complex, by showing that $\d\d^- b=\d\d^+ b$ for every ${q+1}$-dimensional basis element~$b$ with $q\geq 1$. Indeed, let $V_e$ be the set of vertices in~$N_q$ as in the previous paragraph, let $e^+$ be the edge with a source but no target in~$V_e$, let $J_e$ be the set of edges with a target but no source in~$V_e$, and let $e'$ be an input edge in~$N_q$. Since $N_q$~is confluent (Proposition~\ref{P3.9}), there can be a path from~$e'$ to at most one member of~$J_e$, and there is such a path if and only if there is a path from~$e'$ to~$e^+$. The set of vertices in~$N_{q-1}$ corresponding to~$e^+$ is therefore the disjoint union of the sets corresponding to the members of~$J_e$, and it follows from this that $\d\d^- b=\d\d^+ b$ as required.

We define the augmentation as follows: if $a$~is a zero-dimensional basis element then $\e a=1$. From the constellation condition and the linearity of~$N_0$, if $b$~is a $1$-dimensional basis element then $\d^+ b$ and $\d^- b$ are zero-dimensional basis elements, and it follows that $\e\d=0$.

We have now constructed a free augmented directed complex~$K$. From the constellation condition, if $b$~is a positive-dimensional basis element then $\d^+ b$ is a basis element. It therefore follows from Proposition~\ref{P2.8} that $K$~is unital. Since the networks~$N_q$ are acyclic, it also follows from Proposition~\ref{P2.8} that $K$~is loop-free.

By construction, there are no basis elements of dimension greater than~$n$. There is a single basis element~$g$ of degree~$n$, because $N_n$ consists of a single edge. By downward induction on~$q$, the output edge of~$N_q$ corresponds to the basis element $(\d^+)^{n-q}g$. For $0\leq q<n$, it follows that the input edges of~$N_q$ correspond to the terms of the chain $\d^-(\d^+)^{n-q-1}g=(\d^-)^{n-q}g$, and of course the input edge of~$N_n$ corresponds to the single term of the chain~$g$. For $0\leq q<n$ it then follows that the vertices and non-input edges of~$N_q$ correspond to the basis elements $a$~and $\d^+ a$ such that $a$~is a term in $(\d^-)^{n-q-1}g$. From these calculations, we see that $K$~is atomic. Since the thin edges and vertices of the networks~$N_q$ satisfy the conditions of Definition~\ref{D3.3}, the thin basis elements of~$K$ satisfy the conditions of Definition~\ref{D2.5}. Therefore $K$~is an opetopic directed complex.

We will now show that $N_q\cong G_q^-(\langle K\rangle)$. We recall that the edges of $G_q^-(\langle K\rangle)$ correspond to the terms of the chain $g_q^-(\langle K\rangle)$. We have $g_n^-(\langle K\rangle)=g$; also, if $0\leq q<n$, then
$$g_q^-(\langle K\rangle)=(\d^-)^{n-q}g+\d^+ a_1+\ldots+\d^+ a_k$$
where $a_1,\ldots,a_k$ are the terms of $(\d^-)^{n-q-1}g$. From these computations we see that the edges of~$N_q$ correspond to the edges of $G_q^-(\langle K\rangle)$. It clearly follows that the sequence $N_0\to\ldots\to N_n$ is isomorphic to the sequence
$$G_0^-(\langle K\rangle)\to\ldots\to G_n^-(\langle K\rangle),$$
and it is also clear that $K$~is unique up to isomorphism. 

By comparing Definitions \ref{D2.6} and~\ref{D3.4}, we see that $K$~is reduced if and only if all the networks~$N_q$ are reduced; equivalently, $K$~is reduced if and only if the sequence $N_0\to\ldots\to N_n$ is an $n$-dimensional opetope.

This completes the proof.
\end{proof}

\section{Subcomplexes and reductions} \label{S6}

Let $K$ be an opetopic directed complex and let $h$ be a basis element for~$K$; then there is clearly an opetopic directed subcomplex generated by~$h$. In particular, let $K$ be reduced and $n$-dimensional with $n>0$, so that $K$~corresponds to an opetope of positive dimension; then the sources and target for this opetope essentially correspond to the atomic subcomplexes generated by the $(n-1)$-dimensional basis elements of~$K$.
In this way we obtain an algebraic description of source and target opetopes.

This algebraic description seems on the whole to be simpler than the combinatorial description given in \cite{B2},~5. But there is a complication, because a subcomplex of a reduced complex may fail to be reduced. We will now explain how by a two-stage process one can obtain a reduced opetopic directed complex from an arbitrary opetopic directed complex, or, equivalently, an opetope from an opetopic sequence.

In terms of networks, the process is as follows. Let $N_0\to\ldots\to N_n$ be an opetopic sequence. In the first stage, we consider the vertices~$v$ in~$N_q$ such that $v$~is a target only for thin edges. The vertices in~$N_{q-1}$ corresponding to the edges with target~$v$, together with their incident edges, must form a path. We replace the edges in~$N_q$ with target~$v$ by a single thin edge with target~$v$, and we replace the corresponding path in~$N_{q-1}$ by a path consisting of two non-thin edges and a single thin vertex. In this way we obtain a sequence $N'_0\to\ldots\to N'_n$ such that any vertex which is a target only for thin edges is in fact the target for only one edge.

In the second stage we consider the thin edges~$e$ in~$N'_q$ such that $e$~has no target (this can happen only for $q=n$) or such that the target of~$e$ is also the target for at least one non-thin edge. The thin vertex in~$N'_{q-1}$ corresponding to~$e$, together with its incident edges, forms a path with two edges and one vertex. We delete the edge~$e$ from~$N'_q$, and we replace the corresponding path in~$N'_{q-1}$ by a single non-thin edge. In this way we we obtain a sequence of reduced opetopic networks, so that we have an opetope as required. Note that if $N_n$~originally consisted of a thin edge then the final opetope will be $(n-1)$-dimensional.

In terms of complexes this process amounts to forming a quotient of a subcomplex. In terms of bases, we proceed as follows. First, for each basis element~$b$ such that $\d^- b$ is a sum of thin basis elements, we can write the sum as
$$\d^- b=a_1+\ldots+a_k$$
such that
$$\d^-\d^- b=\d^- a_1,\ \d^+ a_1=\d^- a_2,\ \ldots,\ \d^+ a_k=\d^+\d^- b.$$
We obtain the basis for the subcomplex by replacing the basis elements 
$$a_1,\d^+ a_1,a_2,\ldots,a_{k-1},\d^+ a_{k-1},a_k$$
with the single chain $\d^- b$.

Next we obtain the basis for the quotient of this subcomplex in the following way: for each thin basis element~$a$ which is not of the form $\d^- b$ for some other basis element~$b$, we equate~$a$ to~$0$ and we equate $\d^+ a$ to $\d^- a$. In this way we obtain a reduced opetopic directed complex.

In particular the sources and targets for opetopes may be described as follows. Let $K$~be an $n$-dimensional reduced opetopic directed complex with $n>0$, and let $g$ be its $n$-dimensional basis element. If $\d^- g$ is a thin basis element, then $K$~has no sources. If $\d^- g$ is not a thin basis element, then the sources of~$K$ are the reductions of the subcomplexes generated by the terms of $\d^- g$. The target of~$K$ is the reduction of the subcomplex generated by $\d^+ g$.

Note also that it is very easy to describe the bases for the subcomplexes. Indeed, let $L$ be an $m$-dimensional atomic subcomplex of an opetopic directed complex and let $h$ be the $m$-dimensional basis element of~$L$. If $0<r\leq m$ then $(\d^+)^{r-1}h$ is a basis element and $(\d^-)^r h$ can be computed from the formula
$$(\d^-)^r h=\d^-(\d^+)^{r-1}h.$$
For $0\leq q<m$ the $q$-dimensional basis elements of~$L$ are then the terms of the sum
$$g_q^-(\langle L\rangle)=(\d^-)^{m-q}h+\d^+ a_1+\ldots+\d^+ a_k,$$
where $a_1,\ldots,a_k$ are the terms of $(\d^-)^{m-q-1}h$.

We will now carry out these calculations for the reduced opetopic directed complex in Table~1. For the target, we use the subcomplex generated by~$b_{17}$. The basis elements for this subcomplex can be tabulated as follows. 
$$\begin{array}{|l|l|l|l|l|}
\hline
b_{17}&
a_{12},a_{11},a_{10},a_{9.5},a_9,a_8&
a_7,a_6,a_{5.5},a_5&
a_4,a_3,a_2&
a_1\\
\hline
b_{13}&
b_{12},b_{11},b_{10},b_{9.5},b_9,b_8&
b_7,b_6,b_{5.5},b_5&
b_4,b_3,b_2&
\\
\hline
\end{array}$$
The boundaries can be obtained from Table~1. This subcomplex is in fact reduced, so it gives the target opetope directly.

\begin{table}

$$\begin{array}{||l|c|l|l||l|c|l|l||}
\hline\hline
a_{13}&4&a_{12}+b_{16}+b_{15}+b_{14}&b_{13}&b_{13}&3&a_7+a_6+a_5&b_8\\ 
\hline
a_{12}&3&a_7+a_6&b_{12}&b_{12}&2&a_3+a_2&b_6\\
b_{16}&&b_{12}&b_{9.5}&b_{9.5}&&a_3+a_2&b_6\\
b_{15}&&a_5+b_{9.5}&b_9&b_9&&a_4+a_3+a_2&b_5\\
b_{14}&&b_9&b_8&b_8&&a_4+a_3+a_2&b_5\\
\hline
a_7&2&a_3&b_7&b_7&1&b_4&b_3\\
a_6&&a_2+b_7&b_6&b_6&&b_4&b_2\\
a_5&&a_4+b_6&b_5&b_5&&a_1&b_2\\
\hline
a_4&1&a_1&b_4&b_4&0&0&0\\
a_3&&b_4&b_3&b_3&&0&0\\
a_2&&b_3&b_2&b_2&&0&0\\
\hline
a_1&0&0&0&&&&\\
\hline\hline
\end{array}.$$

\caption{The source generated by~$a_{13}$}

\end{table}

For the sources we use the subcomplexes generated by $a_{16}$, $a_{15}$, $a_{14}$ and~$a_{13}$. The corresponding tables of basis elements are as follows.
\begin{align*}
&\begin{array}{|l|l|l|l|l|}
\hline
a_{16}&
a_{9.5}&
b_{12}&
a_3,a_2&
b_4\\
\hline
b_{16}&
b_{9.5}&
b_6&
b_3,b_2&
\\
\hline
\end{array}\\
&\begin{array}{|l|l|l|l|l|}
\hline
a_{15}&
a_{11},a_9&
b_{9.5},a_5&
b_{5.5},a_3,a_2&
a_1\\
\hline
b_{15}&
b_{11},b_9&
b_6,b_5&
b_4,b_3,b_2&
\\
\hline
\end{array}\\
&\begin{array}{|l|l|l|l|l|}
\hline
a_{14}&
a_{10},a_8&
a_{5.5},b_9&
a_4,a_3,a_2&
a_1\\
\hline
b_{14}&
b_{10},b_8&
b_{5.5},b_5&
b_4,b_3,b_2&
\\
\hline
\end{array}\\
&\begin{array}{|l|l|l|l|l|}
\hline
a_{13}&
a_{12},b_{16},b_{15},b_{14}&
a_7,a_6,a_{5.5},a_5&
a_4,a_3,a_2&
a_1\\
\hline
b_{13}&
b_{12},b_{9.5},b_9,b_8&
b_7,b_6,b_{5.5},b_5&
b_4,b_3,b_2&
\\
\hline
\end{array}
\end{align*}
The subcomplexes generated by $a_{16}$, $a_{15}$ and~$a_{14}$ are reduced, so they give the corresponding source opetopes directly. The subcomplex generated by~$a_{13}$ is not reduced, because it contains the thin basis element $a_{5.5}=\d^- a_{10}$ without the basis element~$a_{10}$. To obtain the corresponding reduced complex, we equate $a_{5.5}$~and~$b_{5.5}$ to zero and~$a_4$ respectively. The boundaries in this reduced complex are shown in Table~2.

One can check that the sources constructed here agree with those of \cite{B2}, 5.9--5.13.


\begin{thebibliography}{1}

\bibitem{B1}
J.~ C.~Baez\ and\ J.~Dolan, Higher-dimensional algebra. III. $n$-categories and the algebra of opetopes, Adv. Math. {\bf 135} (1998), no.~2, 145--206. 

\bibitem{B2}
J.~Kock, A.~Joyal, M.~Batanin\ and\ J-F.~Mascari, Polynomial functors and opetopes, Adv. Math. {\bf 224} (2010), no.~6, 2690--2737. 

\bibitem{B3}
T.~Leinster, {\it Higher operads, higher categories}, London Mathematical Society Lecture Note Series, 298, Cambridge Univ. Press, Cambridge, 2004. 

\bibitem{B4}
R.~Steiner, Omega-categories and chain complexes, Homology Homotopy Appl. {\bf 6} (2004), no.~1, 175--200. 

\bibitem{B5}
R.~Steiner, Orientals, in {\it Categories in algebra, geometry and mathematical physics}, 427--439, Contemp. Math., 431 Amer. Math. Soc., Providence, RI, 2007. 

\bibitem{B6}
R.~Steiner, Simple omega-categories and chain complexes, Homology, Homotopy Appl. {\bf 9} (2007), no.~1, 451--465. 

\end{thebibliography}
\end{document}